\documentclass[a4paper,12pt, pdfmx]{article}
\usepackage{hyperref}
\usepackage{bm}
\usepackage{amsmath, amssymb}
\usepackage{amsthm}
\usepackage{mathrsfs}
\usepackage[all]{xy}
\usepackage{enumerate}
\usepackage{comment}
\usepackage[OT2,T1]{fontenc}
\DeclareSymbolFont{cyrletters}{OT2}{wncyr}{m}{n}
\DeclareMathSymbol{\Sha}{\mathalpha}{cyrletters}{"58}

\makeatletter
\@addtoreset{equation}{section}
\makeatother

\usepackage{tikz} %図を描く

\newtheorem{thm}{Theorem}[section]

\newtheorem{prop}[thm]{Proposition}
\newtheorem{lem}[thm]{Lemma}
\newtheorem{cor}[thm]{Corollary}
\theoremstyle{definition}
\newtheorem{rem}[thm]{Remark}

\def \Q{\mathbb{Q}}
\def \Z{\mathbb{Z}}

\DeclareMathOperator{\rank}{rank}

\def \eo{E_{1,s}}
\def \et{E_{2,t}}
\def \rat{\mathscr{F}}

\title{The relationship between face cuboids and elliptic curves}
\author{Takumi Yoshida\\
}
\date{\today}

\begin{document}

\maketitle

\begin{abstract}

A rational face cuboid is a cuboid that all of edges, two of three face diagonals and space diagonal have rational lengths.
%(the term ``face cuboid'' is defined in, for example, \cite{I}).
%%%後でReferenceを書く
%If another face diagonal has integer length, this cuboid is called a perfect cuboid.
%It is expected that there is no perfect cuboid.
We consider an elliptic curve
\[
\eo: y^2=x(x-(2s)^2)(x+(s^2-1)^2)
\]
for a rational number $s \neq 0, \pm 1$,
and define $\tilde{A}$ consisting of all pairs of a rational number $s$ and a non-torsion rational point $(\alpha, \beta ) \in \eo(\Q)$.
We construct a surjective map from $\tilde{A}$ to the set $\rat$ of equivalence classes of rational face cuboids,
and prove that this map is a $32:1$-map.
In this way, we show that the set $\rat$ has infinite elements.
Also, we prove that there are infinitely many $s \in \Q \setminus \{ 0,\pm 1 \}$ with $\rank \eo (\Q)>0$.
In this proof, we construct pairs of $s$ and $(\alpha, \beta) \in \eo (\Q)$ which are not parametric solutions.
%the group $\eo(\Q)$ has infinite element when $t'=5/3$, therefore there are infinite face cuboids up to similarity.
%%%We give the relationship between face cuboids and elliptic curves, and make a bijection induced by this relationship.

%We define a set $A$ by forgetting $y$-coordinates of elements from $\tilde{A}$.
%A certain group $\Gamma$ (the order is $4$, defined in Theorem \ref{onetoone}) acts on the set $A$.
%The map $\tilde{A} \to \{ \text{face cuboids} \}/ \sim$ (where $\sim$ means the similarity of face cuboids) is induced by a bijection from $A/\Gamma$ to a set $B$ defined in Theorem \ref{onetoone}.
%From each element in set $B$, we can get a new elliptic curve $\et$ and its rational point.
%This elliptic curve $\et$ is isomorphic to $E_{1,\tilde{t}'}$ (with $\tilde{t}' \neq t'$) for some $\tilde{t}'$.
%As a result, we get a new element of $\tilde{A}$, and so there are infinitely many numbers $t'$ so that $\eo(\Q)$ has rank $1$ (Corollary \ref{rank}).

\end{abstract}
%1章
%
\section{Introduction}
\label{intro}

A face cuboid is a cuboid that all of edges, two of three face diagonals and space diagonal have integer lengths.
Face cuboids are described in detail in, for example, \cite{l} Section 4.
In this paper, we call a cuboid a ``rational'' face cuboid when all of edges, two of three face diagonals and space diagonal have rational lengths.
Face cuboids and rational face cuboids are essentially the same objects by considering ratios of lengths.

For two rational face cuboids $ABCD\text{-}EFGH$ (as in the following picture) and $A'B'C'D'\text{-}E'F'G'H'$,
we write $ABCD\text{-}EFGH \sim A'B'C'D'\text{-}E'F'G'H'$ if they are similar with correspondence of points, namely there exists $s \in \Q$ such that
\[
BF=sB'F',
EF=sE'F',\text{ and }
GF=sG'F'
\]

%直方体
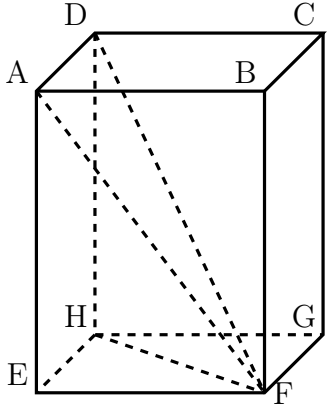
\begin{tikzpicture}[style={very thick}]
\label{jiji}
	\def\x{3}
	\def\y{4}
	\def\z{2}
	\def\gosa{0.25}
	\coordinate (A) at (0,\y,\z);
	\coordinate (B) at (\x,\y,\z);
	\coordinate (C) at (\x,\y,0);
	\coordinate (D) at (0,\y,0);
	\coordinate (E) at (0,0,\z);
	\coordinate (F) at (\x,0,\z);
	\coordinate (G) at (\x,0,0);
	\coordinate (H) at (0,0,0);
	\draw[dashed](D)--(H)--(G);
	\draw[dashed](H)--(E);
	\draw[dashed] (H)--(F);
	\draw (A)--(B)--(C)--(D)--(A)--(E)--(F)--(G)--(C);
	\draw (B)--(F);
	\draw[dashed] (A)--(F);
	\draw[dashed] (D)--(F);
	\foreach \P in {A,...,E,G,H} \draw (\P)+(-\gosa,\gosa,0) node {\P};
	\draw (F)+(\gosa,0,0) node {F};
%	\draw (H) to [out=10, in=80, edge node={node [midway, fill=white] {$\sqrt{x}$}  }] (F);
%	 \draw (A).. controls ($(A)!.2!(B)!10pt!90:(B)$) and ($(A)!.8!(B)!10pt!90:(B)$) .. (B) node [midway, sloped, fill=white] {$b$};
\end{tikzpicture}

We define the set $\rat$ to be the set of the equivalence classes, namely
\[
\rat=\{ \text{rational face cuboids} \} / \sim.
\]

For a rational number $s  \neq 0, \pm 1$, we define an elliptic curve $\eo / \Q$ by
\[
\eo: y^2=x(x-(2s)^2)(x+(s^2-1)^2).
\]
Denoting by $\eo(\Q)_\text{tor}$ the torsion subgroup of $\eo(\Q)$, we define the set $\tilde{A}$ by
\[
\tilde{A}= \left \{ (s,\alpha,\beta) \middle| s \in \Q \setminus \{ 0,\pm 1 \} \text{ and } (\alpha,\beta) \in \eo(\Q) \setminus \eo(\Q)_\text{tor} \right \}.
\]

%\[
%\et: y^2=x(x-(t^2-1)^2)(x+(2t)^2).
%\]
%In this paper, we call a cuboid a rational face cuboid when all of edges, two of three face diagonals and space diagonal have ``rational'' lengths.

%MAIN THM
\begin{thm}
\label{main}
Take $(s,\alpha, \beta) \in \tilde{A}$.
Let
\begin{equation}
\label{eqt}
t=\frac{s \alpha-2s(s^2-1)}{\alpha+2s^2(s^2-1)},
\end{equation}
and
\begin{equation}
\label{eqx}
\gamma=t^2(s-\frac{1}{s})^2.
\end{equation}
Then there exists a rational face cuboid $ABCD\text{-}EFGH$ with $BF=2|t|$, $EF=|t^2-1|$, and $HF=\sqrt{\gamma}$.
%\begin{equation}
%\label{eqy}
%Y=\frac{2t^2(t'^2-1)^2(t'^2+1)^2Y'}{t'(X'+2t'^2(t'^2-1))^2}.
%\end{equation}
\end{thm}

From this theorem, we get a map
\[
\tilde{A} \ni  (s,\alpha, \beta) \mapsto (\text{the equivalent class of }ABCD\text{-}EFGH) \in \rat .
\]

%逆に求める
\begin{thm}
\label{inverse}
%Let $t \neq 0, \pm 1$ and $X$ be rational numbers.
%Suppose $\gamma, \gamma-(t^2-1)^2$, and $\gamma+(2t)^2$ are square numbers which are not $0$.
%There are a rational number $t'$ and a point $(\alpha,\beta) \in \eo(\Q) \setminus \eo(\Q)_{tor}$ which satisfy all assumptions in Theorem \ref{main}.
The map $\tilde{A} \to \rat$ is surjective.

\end{thm}

More precisely, we prove that $\tilde{A} \to \rat$ is $32:1$-map.
The strategy to prove these results is as follows.

We will define the sets $A,B,\tilde{B}$, and $\tilde{B}'$ as follows.
Let $\tilde{A}$ be the set of elements $(s,\alpha,\beta)$ defined just before Theorem \ref{main}.
We define $A$ by
\[
A=\left \{ (s,\alpha) \in (\Q^\times)^{\oplus 2} \middle| \text{there exists $\beta \in \Q^\times$ such that }(s,\alpha, \beta) \in \tilde{A}  \right \}.
\]
Let $\et$ be an elliptic curve $y^2=(x-(t^2-1)^2)(x+(2t)^2)$.
We define $B,\tilde{B}$ and $\tilde{B}'$ by
\[
B=\left \{ (t,\gamma) \in (\Q ^{\times})^ {\oplus 2} \middle| 
\begin{array}{ll}
t \neq \pm 1, \text{ and} \\
\gamma, \gamma-(t^2-1)^2,\gamma+(2t)^2 \in \Q ^{\times} \text{ are all squares} \\
\end{array}
 \right\},
\]
\[
\tilde{B}=\left \{ (t,\gamma,\delta) \middle| t \in \Q \setminus \{ 0, \pm 1\} \text{ and } (\gamma,\delta) \in \et(\Q) \setminus \et(\Q)_\text{tor} \right \},
\]
and
\[
\tilde{B}'=\left \{ (t,\gamma,\delta) \in \tilde{B} \middle| \gamma,\gamma-(t^2-1)^2,\gamma+(2t)^2 \in (\Q^\times)^2 \right \}.
\]
We will define in Section \ref{ell-corr} a certain group $\Gamma$ with order $4$, which acts on $A$.

Define a map $B \to \rat$ as
\[
B \ni  (t, \gamma) \mapsto (\text{the equivalent class of }ABCD\text{-}EFGH) \in \rat,
\]
where $ABCD\text{-}EFGH$ is defined from $(t,\gamma)$ as in Theorem \ref{main}.
We make a map $A/\Gamma \to B$ from the map $A \to B,(s,\alpha) \mapsto (t,\beta)$, and prove that this map is bijective in Theorem \ref{onetoone} (\ref{bijective}).
We show that the map $\tilde{B}' \to B, (t,\gamma,\delta) \mapsto (t,\gamma)$ is $2:1$ in Corollary \ref{Bsurj}.
We show that the map $B \to \rat$ is $4:1$ in Theorem \ref{diagram} (\ref{4to1}).
As a result, we get the following diagram:

%We define $A,B$ as certain subsets of $(\Q^\times)^{\oplus 2}$, and define a certain group $\Gamma \subset \mathrm{Aut}(A)$ with order $4$ in Section \ref{ell-corr}.
%We also define $\tilde{B}, \tilde{B}'$ as certain subsets of $(\Q^\times)^{\oplus 3}$ in Theorem \ref{diagram}, and define maps between these sets as in the following diagram:

\[
\xymatrix{
\tilde{A} \ar[r]^{4:1} \ar[d]_{2:1}       & \tilde{B}' \ar@{^{(}->}[r] \ar[d]^{2:1}       & \tilde{B}\\
A \ar[r]^{4:1}_F \ar[d]_{4:1} &B \ar@{}[lu]|{\circlearrowright} \ar@{}@<-2.2ex>[dl]|{\circlearrowright} \ar[d]^{4:1} \\
A/ \Gamma \ar[ru]_{1:1}& \rat  .
}
\]
%  \ar@/^-20pt/[ll]^{1:1}_{Lemma \  \ref{self}}
%Here, $A,B,\Gamma$ are defined in Section \ref{ell-corr}, and $\tilde{B}', \tilde{B}$ are defined in Theorem \ref{diagram}.
From this diagram, we conclude that the map $\tilde{A} \to \rat$ is $32:1$-map.
For more details, see Theorem \ref{diagram}.

%We will see in Theorem \ref{diagram} that this map is $32:1$-map.

In Section \ref{mainproof}, we determine the torsion group of elliptic curves $\eo$, and prove Theorems \ref{main} and \ref{inverse}.
In Theorem \ref{onetoone}, we give a bijective map $A/\Gamma \to B$ which is induced by $(s,\alpha) \mapsto (t,\gamma)$, where the numbers $s,\alpha,t,\gamma$ are as in Theorem \ref{main}.
In Section \ref{maps}, we define the above maps, and clarify the relationship between them.
Using this map, we prove the infiniteness of rational face cuboids in Corollary \ref{face}.
In Lemma \ref{self}, we define a bijective $\tilde{B} \to \tilde{A}$.
Using this map, we prove that the set $ \left \{ s \in \Q \setminus \{ 0, \pm 1 \}\middle| \rank \eo(\Q)>0 \right \}$ has infinite elements (Corollary \ref{rank}).

\section*{acknowledgement}

The author wishes to thank Professor Masato Kurihara, for giving much helpful advice.

Also the author thanks Hayato Yagi for discussion with him on this elliptic curve.
See Remark \ref{rem} for details.

This work was supported by JST SPRING, Grant Number JPMJST2123.
%もう少し詳しく。

%2章
%
\section{Proof of Theorem \ref{main} and Theorem \ref{inverse}}
\label{mainproof}

In this section, we prove Theorem \ref{main} and Theorem \ref{inverse}.

%torを決定
\begin{lem}
\label{tor}
For any $s \in \Q \setminus \{ 0, \pm 1 \}$, the torsion points $ \eo(\Q)_{tor}$ of $\eo(\Q)$ is determined as
\begin{equation*}
\begin{split}
& \ \ \ \ \eo(\Q)_{tor} \\
&= \{O,(0,0),((2s)^2,0),(-(s^2-1)^2,0),(2s(s+1)^2,\pm 2s(s+1)^2(s^2+1)) \\
&  \ \ \ \ ,(-2s(s-1)^2, \pm 2s(s-1)^2(s^2+1))  \}\\
&\cong \Z/2\Z \times \Z/4\Z.
\end{split}
\end{equation*}
\end{lem}
\begin{proof}
We can easily show that the above 8 points are torsion, which proves the inclusion $\eo(\Q)_{tor} \supset \Z/2\Z \times \Z/4\Z$.
By Mazur's torsion theorem (\cite{M} Theorem 2), the group $\eo(\Q)_{tor}$ is either $\Z/2\Z \times \Z/4\Z$ or $\Z/2\Z \times \Z/8\Z$.
So we just need to prove the non-existence of a point with exact order $8$.
Assume the group is isomorphic to $\Z/2\Z \times \Z/8\Z$.
The point $P=(2s(s+1)^2,2s(s+1)^2(s^2+1)) \in \eo$ is a point with exact order $4$, so all points $Q=(x,y) \in \eo$ with $2Q=P$ has to be rational.
Then, we have
\[
x=2s(s+1)^2 \pm (s^2+1)(s+1)\sqrt{2s} \pm (s+1)\{ (s+1)\sqrt{2s}+2s \} \sqrt{s^2+1}
\]
(with any double sign).
Since these are all rational numbers, the numbers $u:=\sqrt{2s}, v=\sqrt{s^2+1}>0$ are rational.
Define $U$ and $V$ as
\[
U=v+\frac{u^2}{2}(>0), V=uU.
\]
Then, the point $(U,V)$ is on the elliptic curve
\[
E: y^2=x^3-x.
\]
However, we know $E(\Q)=E(\Q)[2]$, and so $V=0$.
This contradicts the fact $u,U>0$.
As a result, the elliptic curve $\eo$ does not have a point with exact order $8$.
Thus we obtain the isomorphism
\[
\eo(\Q)_{tor} \cong \Z/2\Z \times \Z/4\Z.
\]

%もう一つも同様。
\end{proof}

%MAINの証明
\begin{proof}[Proof of Theorem \ref{main}]
First, we prove $\alpha+2s^2(s^2-1) \neq 0$ and $t \neq 0,\pm1$.
The condition $\alpha+2s^2(s^2-1) = 0$ is satisfied if and only if $\alpha=-2s^2(s^2-1)$.
But then, we get
\[
\beta^2=-\{ 2s^2(s^2-1)(s^2+1) \}^2<0,
\]
so it is a contradiction. Thus the number $\alpha+2s^2(s^2-1)$ is not $0$.
If $t \neq 0$, then $\alpha =2(s^2-1)$ by (\ref{eqt}). However, we get 
\[
\beta^2=-\{(s^2-1)(s^2+1)\} ^2 <0,
\]
so it is also a contradiction.
If $t=1$, then we have $\alpha=2s(s+1)^2$, and, by Theorem \ref{tor}, it  contradicts to the assumption that $(\alpha,\beta) \notin E_{1,s}(\Q)_{tor}$.
As the same reason, we get $t \neq -1$.

Next, we prove rational numbers $\gamma, \gamma-(t^2-1)^2$, and $\gamma+(2t)^2$ are square numbers.
The rational numbers $\gamma=t^2(s-\frac{1}{s})^2$ and $\gamma+(2t)^2=t^2(s+\frac{1}{s})^2$ are clearly square.
We have
\begin{equation*}
\begin{split}
\gamma-(t^2-1)^2&=t^2(s-\frac{1}{s})^2-(t^2-1)^2 \\
&=t^2s^2-t^4-1+\frac{t^2}{s^2} \\
&=s^{-2}(s+t)(s-t)(ts+1)(ts-1) \\
&= s^{-2} \cdot  \frac{\{ 2s (s^2-1)(s^2+1) \}^2 \alpha(\alpha-(2s)^2)(\alpha+(\alpha^2-1)^2)}{(\alpha+2s^2(s^2-1))^4} . \\
\end{split}
\end{equation*}
Since the rational number $\alpha(\alpha-(2s)^2)(\alpha+(s^2-1)^2)=\beta^2$ is a square, $\gamma-(t^2-1)^2$ is also a square number.
%From above equations, the point $(X,Y)$ is in $\et(\Q)$.

Finally, we prove  $\gamma, \gamma-(t^2-1)^2,\gamma+(2t)^2 \neq 0$. 
If $\gamma-(t^2-1)^2=0$, then we get $s= 0, \pm 1$ or $\alpha(\alpha-(2s)^2)(\alpha+(s^2-1)^2)=0$ from the above formula.
However, this contradicts the condition on $(s,\alpha,\beta)$.
Thus the number $\gamma-(t^2-1)^2$ is positive because it is a square.
Therefore the inequalities $\gamma,\gamma+(2t)^2>0$ also hold.
\end{proof}

%逆の証明
\begin{proof}[Proof of Theorem \ref{inverse}]
Take $ABCD-EFGH \in \rat$.
We have $BF:EF:BE = 2t:(t^2-1):(t^2+1)$ for some rational number $t$, so we may assume $BF=2t$, $EF=t^2-1$, and $HF = \sqrt{\gamma}$.
By the equation (\ref{eqx}),
\[
s-\frac{1}{s}=\pm \frac{\sqrt{\gamma}}{t},
\]
i.e.
\begin{equation}
\begin{split}
\label{s}
s=\frac{\pm \sqrt{\gamma} \pm \sqrt{\gamma+(2t)^2}}{2t},
\end{split}
\end{equation}
which are all rational numbers since $\sqrt{\gamma+(2t)^2}=DF$.
These numbers are not $0, \pm 1$ since $t,\gamma \neq 0$.
Now
\[
\begin{vmatrix}
s & -2s(s^2-1) \\
1 & 2s^2(s^2-1)
\end{vmatrix}
=2s(s^2+1)(s^2-1)
\neq 0,
\]
therefore there exists $\alpha$ which satisfies the equation  (\ref{eqt}).
%Take the number $Y'$ satisfying (\ref{eqy}).
%For proving that $(X',Y') \notin \eo(\Q)_{tor}$, we only need to make contradiction when $X'$ is $x$-coordinate of the points in Lemma \ref{tor}.
%We can prove it by substituting the values of $X'$ into (\ref{eqt}) and (\ref{eqx}).
%%%%%%これは簡単なので省略する。
\end{proof}

%3章
%
\section{The induced bijection}
\label{ell-corr}

We define the sets $A$ and $B$ by
\[
A= \left \{ (s,\alpha) \in (\Q ^{\times})^ {\oplus 2} \middle| 
\begin{array}{ll}
s \neq \pm 1 \\
(\alpha,\sqrt{\alpha(\alpha-(2s)^2)(\alpha+(s^2-1)^2)}) 
 \in \eo(\Q) \setminus \eo(\Q)_{tor}  \\
\end{array}
\right \},
\]
and
\[
B=\left \{ (t,\gamma) \in (\Q ^{\times})^ {\oplus 2} \middle| 
\begin{array}{ll}
t \neq \pm 1 \\
\gamma, \gamma-(t^2-1)^2,\gamma+(2t)^2 \in \Q ^{\times} \text{ are square numbers} \\
\end{array}
 \right\}.
\]
In this section, we construct a surjection $F: A \to B$, and make a certain induced isomorphism $A/\Gamma \to B$.

%A-->B 4:1
\begin{prop}
\label{four}
Let $F: A \to B$ be the map defined by $(s,\alpha) \mapsto (t,\gamma)$ as in the equation (\ref{eqt}) and (\ref{eqx}).
Then the map $F$ is well-defined and surjective.
\end{prop}

\begin{proof}
The well-definedness of $F$ follows from Theorem \ref{main}.
The surjectivity follows from Theorem \ref{inverse}.
\end{proof}

Define two maps $\sigma, \tau : A \to A$ by
\[
\sigma (s,\alpha) =(-s, -(s^2-1)^2 \cdot \frac{\alpha-(2s)^2}{\alpha+(s^2-1)^2} ),
\]
and
\[
\tau (s,\alpha) = (\frac{1}{s}, -\frac{1}{s^2} \cdot \frac{4(s^2-1)^2}{\alpha}).
\]
It is not difficult to show the well-definedness of these maps.

\begin{thm}
\label{onetoone}

\begin{enumerate}[$(1)$]
\item
The group $\Gamma$ generated by $\sigma$ and $\tau$ is isomorphic to $\Z /2\Z \times \Z /2\Z$,
and acts on $A$.
\item
\label{ind}
For any $\rho \in \Gamma$, we have $F \circ \rho =F$.
\item
\label{bijective}
The induced map
\[
A/\Gamma \to B
\]
is bijective.
\end{enumerate}

\end{thm}

\begin{proof}
%ガンマ不変性の証明,ガンマ不変といっていいか。
\begin{enumerate}[$(1)$]
\item
By direct calculation, we can easily see that $\sigma \circ \sigma = \text{id}_A$, $\tau \circ \tau = \text{id}_A$, and $\sigma \circ \tau = \tau \circ \sigma$,
and so we get $\Gamma \cong \Z /2\Z \times \Z /2\Z$.

\item
By calculation, we get $(F \circ \sigma) (s,\alpha) = F(s,\alpha)$ and $(F \circ \tau) (s,\alpha) = F(s,\alpha)$ for any $(s,\alpha) \in A$.

\item
By Proposition \ref{four}, we have only to prove that the induced map is injective.
Take $(t,\gamma) \in B$.
Then, by (\ref{s}) and (\ref{eqt}), there are exactly four elements in $A$ which are sent to $(t,\gamma) \in B$.
By (\ref{ind}), these four elements form a class in $ A/ \Gamma$.
\end{enumerate}
\end{proof}

%4章
%
\section{Several maps}
\label{maps}

In this section, we construct several maps induced by $F: A \to B$ (in Proposition \ref{four}),
and prove that the map $A \to \rat$ is a $32:1$-map.
Using this result, we prove that there are infinitely many classes of rational face cuboids.
We also prove that there are infinitely many rational numbers $s$ such that $\rank \eo (\Q)>0$.

Let $\et$ be an elliptic curve 
\[
\et: y^2=(x-(t^2-1)^2)(x+(2t)^2).
\]
In order to make the diagram in Section \ref{intro}, we prove the following two lemmas.

\begin{lem}
\label{self}
Let $t \neq 0, \pm 1$ and $(\gamma,\delta) \in \et(\Q)$.
Put
\[
s'=\frac{t+1}{t-1},
\]
\[
\alpha'=\left( \frac{2}{(t-1)^2} \right)^2 \gamma,
\]
and
\[
\beta'=\left( \frac{2}{(t-1)^2} \right)^3 \delta.
\]
Then, we have $(\alpha',\beta') \in E_{1,s'}(\Q)$.
%Especially, elliptic curves $\et$ and $E_{1,s'}$ are isomorphic over $\Q$.
Thus, the map
\begin{equation*}
\begin{split}
\et(\Q) & \to E_{1,s'}(\Q) \\
(\gamma, \delta) & \mapsto (\alpha' , \beta')
\end{split}
\end{equation*}
gives an isomorphism between $\et(\Q)$ and $ E_{1,s'}(\Q)$.
\end{lem}
\begin{proof}
We have

\begin{equation*}
\begin{split}
& \alpha'(\alpha'-(2s')^2)(\alpha'+(s'^2-1)^2) \\
= & \left( \frac{2}{(t-1)^2} \right)^2 \gamma \left( \frac{2}{(t-1)^2} \right)^2 (\gamma - (t^2-1)^2)\left( \frac{2}{(t-1)^2} \right)^2 (\gamma + (2t)^2) \\
= & \left( \frac{2}{(t-1)^2} \right)^6 \delta ^2 \\
= & \beta'^2. \\
\end{split}
\end{equation*}
Therefore, we have $(\alpha',\beta') \in E_{1,s'}(\Q)$.
\end{proof}

%E_{2,tor}に点がない
\begin{lem}
\label{notin}
For any $(t,\gamma) \in B$, we have
\[
(\gamma,\sqrt{\gamma(\gamma-(t^2-1)^2)(\gamma+(2t)^2)}) \notin \et(\Q)_{tor}.
\]
\end{lem}
\begin{proof}
Assume $(\gamma,\sqrt{\gamma(\gamma-(t^2-1)^2)(\gamma+(2t)^2)}) \in \et(\Q)_{tor}$.
%If $\gamma(\gamma-(t^2-1)^2)(\gamma+(2t)^2)=0$, then it contradicts Theorem \ref{main} and Theorem \ref{inverse}.
By Lemma \ref{tor}, we get $\gamma=2t^2(t^2-1)$ or $\gamma=-2(t^2-1)$.

When $\gamma=2t^2(t^2-1)$, the number
\[
\gamma-(t^2-1)^2=(t^2+1)(t^2-1)
\]
is a square.
So the elliptic curve $E:y^2=x(x+1)(x-1)$ has rational point $(t^2, t \sqrt{(t^2+1)(t^2-1)})$ over $\Q$.
However, we have the equation $E(\Q)=\{O, (0,0), (\pm 1, 0)  \}$, and this contradicts the assumption $t \neq 0, \pm 1$.

In the same way, the equation $\gamma=-2(t^2-1)$ does not hold.

From the above, we get $(\gamma,\sqrt{\gamma(\gamma-(t^2-1)^2)(\gamma+(2t)^2)}) \notin \et(\Q)_{tor}$.
\end{proof}

\begin{cor}
\label{Bsurj}
We define $\tilde{B}'$ by
\[
\tilde{B}'=
\left \{ (t,\gamma,\delta) \in (\Q ^{\times})^ {\oplus 3} \middle| 
\begin{array}{ll}
t \neq \pm 1 \\
\gamma, \gamma-(t^2-1)^2,\gamma+(2t)^2 \in \Q ^{\times} \text{ are square numbers} \\
(\gamma,\delta) \in \et(\Q)  \setminus \et(\Q)_\text{tor} \\
\end{array}
 \right\}.
\]
Then the natural map $\tilde{B}' \to B, (t,\gamma,\delta) \mapsto (t, \gamma)$ is surjective.
\end{cor}

\begin{proof}
Take any element $(t,\gamma) \in B$.
Let $\delta = \sqrt{\gamma(\gamma-(t^2-1)^2)(\gamma+(2t)^2)}$.
Since the numbers $\gamma, \gamma-(t^2-1)^2,\gamma+(2t)^2$ are square, the number $\delta$ is rational.
By Lemma \ref{notin}, we have $(\gamma, \delta) \in \eo(\Q) \setminus \eo(\Q)_\text{tor}$.
Therefore, the element $(t,\gamma, \delta)$ is in the set $\tilde{B}'$.
\end{proof}

\begin{thm}
\label{diagram}

\begin{enumerate}[$(1)$]
\item
\label{4to1}
We define $B$ by
%\[
%\tilde{A}= \left \{ (t',X',Y') \middle| t'\neq 0,\pm 1 \text{ and } (X',Y') \in \eo(\Q) \setminus \eo(\Q)_\text{tor} \right \},
%\]
\[
\tilde{B}=\left \{ (t,\gamma,\delta) \middle| t \in \Q \setminus \{ 0, \pm 1\} \text{ and } (\gamma,\delta) \in \et(\Q) \setminus \et(\Q)_\text{tor} \right \}.
\]
Then we have the following commutative diagram:
\[
\xymatrix{
\tilde{A} \ar[r]^{4:1} \ar[d]_{2:1}       & \tilde{B}' \ar@{^{(}->}[r] \ar[d]^{2:1}       & \tilde{B} \\ % \ar@/^-20pt/[ll]^{1:1}_{Lemma \  \ref{self}}\\
A \ar[r]^{4:1}_F \ar[d]_{4:1} &B \ar@{}[lu]|{\circlearrowright}  \ar@{}@<-2ex>[dl]|{\circlearrowright} \ar[d]^{4:1} \\
A/ \Gamma \ar[ru]_{1:1} & \rat
},
\]
where the correspondences of elements in these maps are as follows:
\[
\xymatrix{
(s,\alpha,\beta) \ar@{|->}[r]^{4:1} \ar@{|->}[d]_{2:1}       &(t,\gamma,\delta) \ar@{|->}[r] \ar[d]^{2:1}       & (t,\gamma,\delta) \\ % \ar@/^-20pt/[ll]^{1:1}_{Lemma \  \ref{self}}\\
(s,\alpha) \ar@{|->}[r]^{4:1}_F \ar@{|->}[d]_{4:1} &(t,\gamma) \ar@{}[lu]|{\circlearrowright}  \ar@{}@<-2ex>[dl]|{\circlearrowright} \ar@{|->}[d]^{4:1} \\
(s,\alpha) \ar@{|->}[ru]_{1:1} & ABCD\text{-}EFGH
}
\]
%where the symbol ``$\equiv$'' means the congruence of cuboids with correspondence of points in the picture in Theorem \ref{main} (\ref{sec}),
where the map $B \to \rat , (t,\gamma) \mapsto (\text{a rational face cuboid }ABCD\text{-}EFGH)$ is defined as in Theorem \ref{main}.

\item
\label{bij2}
The map $\tilde{B} \to \tilde{A},(t,\gamma,\delta) \mapsto (s', \alpha', \beta')$ as in Lemma \ref{self} is bijective.
\item
\label{nots}
Suppose that $\tilde{B} \to \tilde{A}$ is the map in (\ref{bij2}).
Then the composition $\tilde{A} \to \tilde{B}' \to \tilde{B} \to \tilde{A}$ is not the identity map.

\end{enumerate}
\end{thm}

\begin{proof}  %[proof of \rm{Theorem \ref{diagram}}]
\begin{enumerate}[$(1)$]
\item
By Theorem \ref{onetoone} (\ref{bijective}) and Corollary \ref{Bsurj}, we have only to show that the map $B \to \rat$ is a $4:1$ map.
The surjectivity is clear.
In order for $(t_1,\gamma_1),(t_2,\gamma_2) \in B$ to define the congruent cuboid, it is necessary and sufficient that
\[
2|t_1|:|t_1^2-1|:\sqrt{\gamma_1}=2|t_2|:|t_2^2-1|:\sqrt{\gamma_2},
\]
i.e. $(t_2,\gamma_2)=(\pm t_1,\gamma_1),(\pm \frac{1}{t_1},\frac{\gamma_1}{t_1^2})$.
Thus, the map is $4:1$.

\item
It is clear by Lemma \ref{self}.
\item
By Theorem \ref{main} and Lemma \ref{self}, the composition $\tilde{A} \to \tilde{B}' \to \tilde{B} \to \tilde{A}$ maps $(s, \alpha, \beta)$ to an element whose $s$-coordinate is $\frac{ (s+1)\alpha+2s(s-1)^2(s+1)}{(s-1)\alpha-2s(s-1)(s+1)^2}$.
\end{enumerate}
\end{proof}

\begin{rem}
\label{rem}
%\begin{enumerate}[$(1)$]
%\item
%\label{rem1}
%\item
%\label{rem2}
Hayato Yagi studied in his master's course elliptic curves $E_{a,b}:y^2=x(x-a^2)(x+b^2)$ such that $a$ and $b$ are rational numbers so that $a^2+b^2$ is a square number.
%He calls this elliptic curve ``Frey curve with degree $2$''.
%%%%自信ない！！！
This elliptic curve is isomorphic to $\eo$ when $a:b=2s:(s^2-1)$. In this sense, elliptic curves $E_{a,b}$ and $\eo$ are essentially the same.
He determined the torsion points of $E_{a,b}(\Q)$.
The author's proof of Lemma \ref{tor} is based on his proof.
He also gave the author many examples of $a$ and $b$ such that the rank of $E_{a,b}(\Q)$ is $1$.
For example, he found that the elliptic curve $\eo (\Q )$ has rank 1 when $s=5/3$.
This result is used in the proof of the following two corollaries.
%\end{enumerate}
\end{rem}

%face cuboids無限個
\begin{cor}
\label{face}
There are infinitely many kinds of rational face cuboids up to isomorphism.
\end{cor}
\begin{proof}
For example, when $s=5/3$, the point $P=(\alpha,\beta)=(-20/27,1120/243)$ is the rational point of $\eo$ which is not torsion.
Therefore the point $P$ has infinite order in $\eo$ (in fact, this elliptic curve has rank $1$).
For any integer $n \geq 1$, let $(\alpha_n,\beta_n)=[n]P$.
Then, the elements $F(s,\alpha_n) \in B$ are different.
By Theorem \ref{diagram}, the map $B \to  \rat $ is $4:1$, so there are infinite elements in $\rat$.
\end{proof}

\begin{cor}
\label{rank}
The set $ \left \{ s \in \Q \setminus \{ 0, \pm 1 \}\middle| \rank \eo(\Q)>0 \right \}$ has infinite elements.
%There are infinitely many rational numbers $s$ with $\rank \eo(\Q)>0$.
\end{cor}
\begin{proof}
Let $s=5/3$ and define $\alpha_n$ and $\beta_n$ the rational numbers as in the proof of Corollary \ref{face}.
The composition $\tilde{A} \to \tilde{B}' \to \tilde{B} \to \tilde{A}$ in Theorem \ref{diagram} (\ref{nots}) maps $(s, \alpha_n, \beta_n)$ to elements whose $s$-coordinates are
\[
\frac{ (s+1)\alpha_n+2s(s-1)^2(s+1)}{(s-1)\alpha_n-2s(s-1)(s+1)^2},
\]
as we saw in the proof of Theorem \ref{diagram} (\ref{nots}).
%By the equation (\ref{eqt}), $t$-coordinates of the images of $(s,\alpha_n,\beta_n)$ under the map $\tilde{A} \to \tilde{B}$ are $\{s \alpha_n-2s(s^2-1)\} / \{\alpha_n+2s^2(s^2-1)\}$, and thus they are different.
%Hence, by Lemma \ref{self}, on the images of them under the composition $\tilde{A} \to \tilde{B}' \to \tilde{B} \to \tilde{A}$, $s$-coordinates are different.
As a result, there are infinitely many $s$ so that $\eo$ has points with infinite order.
\end{proof}


\begin{thebibliography}{99}
%\bibitem{I} http://unsolvedproblems.org/S59.pdf

\bibitem{l} Leech, J. ``The rational cuboid revisited'', Amer. Math. Monthly 84 (1977), 518-533.
%Leech, J. (1977). The rational cuboid revisited. Amer. Math. Monthly 84, 518\UTF{2013}533.

\bibitem{M} Mazur, B. ``Rational isogenies of prime degree (with an appendix by D. Goldfeld)'', In: Invent. Math. 44.2 (1978), 129-162.


%J. Choi, On the 2-adic valuations of central L-values of elliptic curves, J. Number
%Theory 204 (2019), 405\UTF{2013}422. MR3991426





\end{thebibliography}
\end{document}